\documentclass[11pt,reqno,oneside,a4paper]{amsart}

\usepackage{graphicx}
\usepackage{amssymb}
\usepackage{epstopdf}

\usepackage{amsmath,amsfonts,amsthm,mathrsfs,amssymb,cite}
\usepackage[usenames]{color}

\newtheorem{thm}{Theorem}[section]

\newtheorem{cor}[thm]{Corollary}
\newtheorem{lem}{Lemma}[section]
\newtheorem{prop}{Proposition}[section]
\theoremstyle{definition}
\newtheorem{defn}{Definition}[section]

\theoremstyle{remark}

\newtheorem{rem}{Remark}[section]
\numberwithin{equation}{section}

\numberwithin{equation}{section}

\newcounter{saveeqn}

\title[Determination of a fractional Helmholtz equation]{Determining a fractional Helmholtz system with unknown source and medium parameter}

\author{Xinlin Cao}
\address{Department of Mathematics, Hong Kong Baptist University, Kowloon, Hong Kong SAR, China}
\email{xlcao.math@foxmail.com}

\author{Hongyu Liu}
\address{Department of Mathematics, Hong Kong Baptist University, Kowloon, Hong Kong SAR, China}
\email{hongyu.liuip@gmail.com, hongyuliu@hkbu.edu.hk}

\date{} 
\begin{document}
\maketitle

\begin{abstract}

We are concerned with an inverse problem associated with the fractional Helmholtz system that arises from the study of viscoacoustics in geophysics and thermoviscous modelling of lossy media. We are particularly interested in the case that both the medium parameter and the internal source of the wave equation are unknown. Moreover, we consider a general class of source functions which can be frequency-dependent. We establish several general uniqueness results in simultaneously recovering both the medium parameter and the internal source by the corresponding exterior measurements. In sharp contrast, these unique determination results are unknown in the local case, which would be of significant importance in thermo- and photo-acoustic tomography.

\medskip

\noindent{\bf Keywords:}~~Fractional Helmholtz equation, simultaneous recovery, low-frequency asymptotics, compact embedding theorem, strong uniqueness property, Runge approximation property

\noindent{\bf 2010 Mathematics Subject Classification:}~~35R30, 26A33, 35J05, 35P25

\end{abstract}

\tableofcontents

\section{Introduction}

\subsection{Mathematical setup}

Let $\Omega$ be a bounded Lipschitz domain in $\mathbb{R}^n$, $n\geq 2$. Let $q(x)\in L^\infty(\Omega)$ and $\sigma(x) \in C^\infty(\mathbb{R}^n)$ be a symmetric-positive-definite matrix-valued function such that there exists $\lambda\in (0, 1)$,
\[
\lambda|\xi|^2\leq \xi^T\sigma(x)\xi\leq \lambda^{-1}|\xi|^2\quad \mbox{for all}\ \ x\in\Omega\ \ \mbox{and}\ \ \xi\in\mathbb{R}^n.  
\]
Let $\omega\in\mathbb{R}_+$ and $p(x,\omega)$, $x\in\Omega$, be an $L^2$ function for any fixed $\omega$. Let $s\in(0, 1)$ be a fractional order. Consider the following fractional-order Helmholtz system for $u(x, \omega)\in H^s(\mathbb{R}^n)$ associated with a fixed frequency $\omega\in\mathbb{R}_+$,
\begin{equation}\label{eq:fractional3}
\begin{cases}
&\displaystyle{(-\nabla\cdot(\sigma\nabla))^s{u}(x,w)+\omega^2q(x){u}(x,w)=p(x,\omega), \quad x\in\Omega},\medskip\\
&\displaystyle{{u}(x,w)=\psi(x), \quad x\in\Omega_e:=\mathbb{R}^n\backslash\overline{\Omega}}, 
\end{cases}
\end{equation}
where $\psi\in H^s(\mathbb{R}^n)$. 
It is assumed that $0$ is not an eigenvalue of the nonlocal operator $\left(-\nabla\cdot(\sigma\nabla)\right)^s+\omega^2q$, which means that if
\begin{equation}\label{condition}
\begin{cases}
&\displaystyle{ \left(-\nabla\cdot(\sigma\nabla)\right)^s {u}+\omega^2 q {u}=0\quad\mbox{in}\ \ \Omega},\medskip\\
& {u}=0\quad\mbox{in}\ \ \Omega_e. 
\end{cases}
\end{equation}
then we must have $u\equiv0 \in \mathbb{R}^n$. Under the aforesaid condition, the well-posedness of \eqref{eq:fractional3} shall be verified in what follows, and in particular, one has a well-defined Dirichlet-to-Neumann (DtN) map associated with \eqref{eq:fractional3} as follows,
\begin{equation}\label{eq:dtn2}
\Lambda^\omega_{q, p}(\psi)=\left(-\nabla\cdot(\sigma\nabla)\right)^s {u}(x,\omega)|_{x\in\Omega_e},
\end{equation}
where $u$ is the solution to \eqref{eq:fractional3}. 

In this paper, we are mainly concerned with the inverse problem of recovering both $p$ and $q$ by knowledge of the DtN operator $\Lambda^\omega_{q,p}$; that is
\begin{equation}\label{eq:ip1}
\Lambda^\omega_{q,p}\longrightarrow (\Omega; q, p)
\end{equation}
Throughout, we assume that the elliptic tensor $\sigma$ is known. Physically speaking, $p$ represents a source term and $q$ represents the medium parameter of an inhomogeneous medium supported in $\Omega$. $\Lambda_{q,p}^\omega$ encodes the exterior measurements. The major findings of this paper can be first roughly described as follows. The multiple-frequency data $(\psi, \Lambda^\omega_{p, q}(\psi))$ for any fixed $\psi\in H^s(\mathbb{R}^n)$ and $\omega\in (0, \omega_0)$ with any fixed $\omega_0\in\mathbb{R}_+$ uniquely determine a generic source term $p(x,\omega)$ if it satisfies certain ``regularity" condition with respect to $\omega$. 
Moreover, the aforementioned recovery result of the source term is independent of the medium parameter $q$. It is remarked that the result includes the special case with $\psi\equiv 0$. In such a case, the exterior measurement data are the Neumann data generated solely due to the internal source in $\Omega$, and our unique determination result indicates that one can recover the source term by using such data without knowing the surrounding medium $(\Omega, q)$. As for the recovery of the medium parameter $q$, we need to make use of the multiple measurement data $(\psi, \Lambda^\omega_{p, q}(\psi))$ with all $\psi\in H^s(\mathbb{R}^n)$ and $\omega\in (0, \omega_0)$. 

Next, we detail the major recovery results by stating the following two theorems. 

\begin{thm}\label{thm:1}

Let $\Omega\subset\mathbb{R}^n$ be a bounded Lipschitz domain and $\mathcal{O}_1$, $\mathcal{O}_2\subset\mathbb{R}^n\backslash\overline{\Omega}$ be two arbitrary nonempty open subsets. Let $\omega_0\in\mathbb{R}_+$ be fixed, and let $p_j, q_j, \sigma$, $j=1, 2$, be as described earlier. It is further assumed that $p_j(x,\omega)\in C^3[0,\omega_0]$ with respect to $\omega$ for any $x\in\Omega$.
If for any fixed $\psi\in C_c^{\infty}(\mathcal{O}_1)$, there holds
\begin{equation}\label{eq:a2}
\Lambda_{q_1, p_1}^\omega(\psi)|_{\mathcal{O}_2}=\Lambda_{q_2,p_2}^\omega(\psi)|_{\mathcal{O}_2}\quad\mbox{for}\ \ \omega\in (0, \omega_0),
\end{equation}
then one has
\begin{equation}\label{eq:a3}
\frac{\partial^j p_1}{\partial\omega^j}(x,\omega)\bigg|_{\omega=0}=\frac{\partial^j p_2}{\partial\omega^j}(x,\omega)\bigg|_{\omega=0},\ \ \ x\in\Omega,\ \ j=0, 1. 
\end{equation}

\end{thm}

\begin{thm}\label{thm:2}

Let $\omega_0$ and $p_j, q_j, \sigma$, $j=1, 2$, be described in Theorem~\ref{thm:1}. $\mathcal{O}_1, \mathcal{O}_2\subset\mathbb{R}^n\backslash\overline{\Omega}$ are two arbitrary nonempty open sets.
If there hold
\begin{equation}\label{eq:a1}
\frac{\partial^2 (p_1-p_2)}{\partial\omega^2}(x,\omega)\bigg|_{\omega=0}=0,\ \ \ x\in\Omega,
\end{equation}
and
\begin{equation}\label{eq:a2}
\Lambda_{q_1, p_1}^\omega(\psi)|_{\mathcal{O}_2}=\Lambda_{q_2,p_2}^\omega(\psi)|_{\mathcal{O}_2}\quad\mbox{for any}\ \ \psi\in C_c^{\infty}(\mathcal{O}_1)\ \ \mbox{and}\ \ \omega\in (0, \omega_0),
\end{equation}
then one has
\begin{equation}\label{eq:a3}
q_1=q_2. 
\end{equation}

\end{thm}

\begin{rem}\label{rem:1}
It is remarked that the recovery of the internal source function $p$ in Theorem~\ref{thm:1} is independent of the medium parameter $q$. Similarly, the recovery of the medium parameter $q$ in Theorem~\ref{thm:2} is independent of the internal source function $p$. 
\end{rem}

\begin{cor}\label{cor:1}
Suppose that $p(x,\omega)=\eta(x)\omega+\zeta(x)$ with $\eta, \zeta\in L^2(\Omega)$. Then  both $p(x,\omega)$ and $q(x)$ can be uniquely determined by knowledge of $\left(\psi, \Lambda_{q, p}^\omega(\psi)\right)$ for all $\psi\in H^s(\mathbb{R}^n)$ and $\omega\in (0, \omega_0)$. 
\end{cor}

\subsection{Motivation and background}

The major motivation of the current article comes from the study of the fractional wave equation 
\begin{equation}\label{eq:fractional11}
\begin{cases}
& \displaystyle{ \frac{1}{c(x)^2} \frac{\partial^2 u}{\partial t^2}(x,t)+\left(-\nabla\cdot(\sigma\nabla) \right)^s u(x, t)=h(x,t),\quad (x,t)\in \Omega\times \mathbb{R}_+,}\smallskip\\
& \displaystyle{u(x, t)\big|_{t=0}=f(x),\quad \frac{\partial u}{\partial t}\bigg|_{t=0}=g(x),\quad x\in\Omega. }
\end{cases}
\end{equation} 
In \eqref{eq:fractional11}, $c$ is a positive function that signifies the wave speed, and $f, g, h$ are source functions. 
The fractional wave equation arises from the study of viscoacoustics in geophysics and thermoviscous modelling of lossy media; see e.g. \cite{Chen02} and \cite{Chen03} and the references therein for related research. Introduce the temporal Fourier transform,
\begin{equation}\label{eq:ft1}
\hat{\gamma}(x,\omega)=\int_{0}^\infty \gamma(x, t) e^{-\mathrm{i}\omega t}\ dt. 
\end{equation}
Assuming that the temporal Fourier transforms exist for both $u$ and $h$ in \eqref{eq:fractional11}, and applying to both sides of the equation, one has by straightforward calculation the following fractional Helmholtz equation,
\begin{equation}\label{eq:ff22}
 \left(-\nabla\cdot(\sigma\nabla)\right)^s \hat{u}(x,\omega)-\frac{\omega^2}{c^2} \hat{u}(x,\omega)=\frac{\mathrm{i}\omega}{c^2} f(x)+\frac{1}{c^2}g(x)+\hat{h}(x,\omega)\quad \mbox{in}\ \ \Omega. 
\end{equation}
The fractional Helmholtz system \eqref{eq:fractional3} is obviously a generalized one of \eqref{eq:ff22}. Hence, if one considers the inverse problem associated with the fractional wave equation of recovering the unknown sound speed and the internal sources, in certain scenarios, it can actually be reduced to studying the inverse problem \eqref{eq:ip1} associated with \eqref{eq:fractional3}. 

The study of inverse problems associated to the fractional Schr\"odinger using measurements encoded in the exterior DtN map has received growing interest recently \cite{Mikko16,Ghosh17,GRSU,Cao17,hl17,lai2017global,RS17,RS172}. In the literature, the inverse problems are usually referred to as the Calder\'on problems for the fractional PDEs, a name that is inherited from the corresponding study of non-fractional (or so-called nonlocal) PDEs. However, in all of the aforementioned literatures on the fractional Calder\'on problems, the studies therein are mainly concerned with the recovery of the medium parameters, namely the unknown coefficients of the underlying fractional PDEs, and there is no result concerning the simultaneous recovery of the internal sources. On the other hand, the simultaneous recovery of the medium parameters and the internal sources of a bounded body from the exterior measurements in the nonlocal case has also received significant attentions recently in the literature due to its practical importance. In thermoacoustic and photoacoustic tomography \cite{W09}, the simultaneous recovery of a sound speed and an internal source term from the corresponding boundary measurements was considered in \cite{KM,LU15,FH}. Therein, the governing equation is the local Helmholtz equation of the form \eqref{eq:ff22} with $s=1$, $\sigma=I_{3\times 3}$ and $g=\hat h\equiv 0$, and one intends to recover both $c$ and $f$ from the multiple frequency boundary DtN map. Though it is believed that one can establish the simultaneous recovery result in a certain general scenario, the uniqueness results in \cite{KM,LU15,FH} were only established for certain restrictive configurations. We would also like to mention in passing that the simultaneous recovery of unknown medium parameters and internal sources were also considered for the electromagnetic wave phenomena arising from brain imaging \cite{DHU:17} and geomagnetic detection technique \cite{DLL}. It is readily observed that in the exterior measurement, the information of the unknown sound speed is coupled with that of the unknown source, and this makes the simultaneous recovery radically challenging. Indeed, the instability of the linearized problem for the aforementioned tomography problem was shown in \cite{SU13}. In this article, as a direct application of the general uniqueness results established in Theorems~\ref{thm:1} and \ref{thm:2}, we show that in the nonlocal setting, one can establish much more general simultaneous recovery results for the fractional wave equation \eqref{eq:ff22}. We shall discuss those results at the end of Section~\ref{sect:5}. 

Finally, we briefly discuss the mathematical arguments and technical novelties for the current study. Our mathematical techniques are inspired by \cite{LU15} mentioned earlier by one of the authors of this article for the simultaneous recovery in theormoacoustic and photoacoustic tomography. In \cite{LU15}, the low-frequency asymptotics of the wave fields is used to extract the information of the sound speed and the internal source, respectively. The asymptotic expansion is based on a Lippmann-Schwinger integral equation. However, for the fractional Helmholtz equation, one does not have a similar integral representation of the wave field. For the low-frequency asymptotics in the current study, we develop a variational argument together with a compact embedding theorem for the fractional Sobolev space. The asymptotic expansion of the wave fields enables us to extract from the exterior measurement data the information of the medium parameter and the internal sources, respectively. Then we repeatedly make use of the strong uniqueness principe and the Runge approximation property for the fractional Laplacian to recover the medium parameter and the source function, respectively. 

The rest of the paper is organized as follows. In Section 2, we present some preliminary knowledge on the fractional Sobolev spaces and the fractional differential operator. Section 3 is devoted to the study of the well-posedness of the fractional Helmholtz system. In Section 4, we derive the low-frequency asymptotics of the wave fields. Section 5 presents the major unique recovery results for the inverse problem.

\section{Preliminary knowledge}

In this section, for the use in our subsequent study, we present some preliminary knowledge on fractional Sobolev spaces as well as the definition and main properties of the fractional nonlocal operator $(-\triangledown\cdot(\sigma\triangledown))^s$.

\subsection{Fractional Sobolev spaces}

For any $s\in (0,1)$, the fractional Sobolev space $H^s(\mathbb{R}^n)=W^{s,2}(\mathbb{R}^n)$ is the standard $L^2$ based Sobolev space endowed with the norm
\begin{equation*}
\lVert u\rVert_{H^s(\mathbb{R}^n)}=\lVert u \rVert_{L^2(\mathbb{R}^n)}+\lVert (-\vartriangle)^{\frac{s}{2}}u\rVert_{L^2(\mathbb{R}^n)},
\end{equation*} 
which can also be expressed equivalently as 
\begin{equation}\label{normdef}
\lVert u\rVert_{H^s(\mathbb{R}^n)}=\lVert u \rVert_{L^2(\mathbb{R}^n)}+\lvert u\lvert_{H^s(\mathbb{R}^n)},
\end{equation}
where $\lvert u\lvert_{H^s(\mathbb{R}^n)}$ is the corresponding semi-norm in $H^s(\mathbb{R}^n)$ and
\begin{equation}\label{seminorm}
	\lvert u\lvert_{H^s(D)}:=\left( \int_{D\times D}\frac{|u(x)-u(y)|^2}{|x-y|^{n+2s}}dxdy\right) ^{\frac{1}{2}}
\end{equation}
for any open subset $D\subset\mathbb{R}^n$.

Suppose $U\subset\mathbb{R}^n$ is an open set. Following the notation in \cite{Mikko16}, we introduce the following fractional Sobolev spaces,
\begin{align*}
H^s(U)&:=\left\lbrace u|_{U}; u\in H^s(\mathbb{R}^n)\right\rbrace, \\
\tilde{H}^s(U)&:=\mbox{ closure of } C^{\infty}_c(U) \mbox{ in }
H^s(\mathbb{R}^n),\\
H^s_0(U)&:=\mbox{ closure of } C^{\infty}_c(U) \mbox{ in }
H^s(U),\\
H^s_{\bar{U}}(\mathbb{R}^n)&:=\left\lbrace u \in H^s(\mathbb{R}^n); \mbox{supp}(u)\subset\bar{U}\right\rbrace ,
\end{align*}
where $H^s(U)$ is a Hilbert space equipped with the following norm
\begin{equation*}
\lVert u\rVert_{H^s(U)}:= \mbox{inf}\left\lbrace \lVert w\rVert_{H^s(\mathbb{R}^n)}; w \in H^s(\mathbb{R}^n), w|_{U}=u \right\rbrace .
\end{equation*} 
It is known that $\tilde{H}^s(U)=H^s_{\bar{U}}\subseteq H^s_0(U)$ if $U$ is a Lipschitz domain. For more detailed properties of the fractional Sobolev spaces, we refer to \cite{Nezza12}, \cite{Mclean00} and \cite{Triebel02}.

\subsection{Definition of $(-\nabla\cdot(\sigma\nabla))^s$}

Let $\sigma \in C^\infty(\mathbb{R}^n)$ be introduced earlier that is a symmetric-positive-definite matrix-valued function satisfying the ellipticity condition. It can be verified that the Laplacian operator $L_{\sigma} := -\nabla\cdot(\sigma\nabla)$ is a linear, densely-defined, self-adjoint second order partial differential operator\cite{Stinga10}. The fractional Laplacian operator $L_{\sigma}^s := (-\nabla\cdot(\sigma\nabla))^s, s\in (0,1)$ can be defined in a spectral way as
\begin{equation*}
L_{\sigma}^s=(-\nabla\cdot(\sigma\nabla))^s:=\int_{0}^{\infty}\lambda^sdE(\lambda)=\frac{1}{\Gamma(-s)}\int_{0}^{\infty}(e^{-tL_{\sigma}}-Id)\frac{dt}{t^{1+s}},
\end{equation*} 
where $\left\lbrace E(\lambda)\right\rbrace $ is the unique spectral resolution of $L_{\sigma}=-\nabla\cdot(\sigma\nabla)$ and $e^{-tL_{\sigma}} = \int_{0}^{\infty}e^{-t\lambda}dE(\lambda)$, $t\geq0$ with domain $L^2(\mathbb{R}^n)$ is the heat-semigroup generated by $L_{\sigma}$. It is essential that this heat-semigroup can be represented by integration against a nonnegative symmetric heat kernel $\mathscr{P}_t(x,y)$ as 
\begin{equation*}
e^{-tL_{\sigma}}f(x)=\int_{\mathbb{R}^n}\mathscr{P}_t(x,y)f(y)dy, \quad x\in \mathbb{R}^n,
\end{equation*} 
for any $f\in L^2(\mathbb{R}^n)$.

Define 
\begin{equation*}
\mathscr{K}_s(x,y):=\frac{1}{2|\Gamma(-s)|}\int_{0}^{\infty}\mathscr{P}_t(x,y)\frac{dt}{t^{1+s}}
\end{equation*}
as the kernel of the heat-semigroup and it has the following pointwise estimate \cite[Theorem 2.4]{Caffa16}:
\begin{equation}\label{estimate}
\frac{c_1}{|x-y|^{n+2s}}\leq \mathscr{K}_s(x,y)\leq \frac{c_2}{|x-y|^{n+2s}}, \quad x,y \in \mathbb{R}^n
\end{equation}
for some constants $c_1>0$ and $c_2>0$.

Using \cite[Theorem2.4]{Caffa16}
again, we can obtain that
\begin{equation}\label{integralrepre}
(L_{\sigma}^sv,w)_{\mathbb{R}^n}=\int_{\mathbb{R}^n\times\mathbb{R}^n}(v(x)-v(y))(w(x)-w(y))\mathscr{K}_s(x,y)dxdy
\end{equation}
 for any $v,w \in H^s(\mathbb{R}^n)$. Since $\mathscr{P}_t(x,y)$ is symmetric, we know that $\mathscr{K}_s(x,y)$ is also symmetric and $(L_{\sigma}^sv,w)_{\mathbb{R}^n}=(L_{\sigma}^sw,v)_{\mathbb{R}^n}, \forall v,w\in H^s(\mathbb{R}^n)$.
 The readers can refer to \cite{Ghosh17} for more discussions about the fractional operator $L_{\sigma}^s=(-\nabla\cdot(\sigma\nabla))^s$.

\section{Nonlocal problem for the fractional Helmholtz system}

In this section, we study the nonlocal Dirichlet problem for the fractional Helmholtz equation \eqref{eq:fractional3}. First of all, we briefly discuss the well-posedness of this problem.

\subsection{Well-posedness}

Let $\Omega$ be a bounded Lipschitz domain in $\mathbb{R}^n$. Associated with the problem \eqref{eq:fractional3}, we define the following bilinear form 
\begin{align}\label{def:bilinear}
\mathscr{B}_{\omega,q}(v,w)=&(L_{\sigma}^sv,w)_{\mathbb{R}^n}+\omega^2\int_{\Omega}q(x)v(x)w(x)dx\notag\\
=&\int_{\mathbb{R}^n\times\mathbb{R}^n}(v(x)-v(y))(w(x)-w(y))\mathscr{K}_s(x,y)dxdy\notag\\
&+\omega^2\int_{\Omega}q(x)v(x)w(x)dx
\end{align}
for any $v,w\in H^s(\mathbb{R}^n)$. By our earlier discussion, it is clear that $\mathscr{B}_{\omega,q}$ is symmetric and bounded. Next, we utilize $\mathscr{B}_{\omega,q}$ to give the definition of the weak solution for our problem.

\begin{defn}
	Let $\Omega\subset\mathbb{R}^n$ be a bounded Lipschitz domain. Fix any $\omega_0\in\mathbb{R}_+$, given $\omega\in (0,\omega_0)$, $p(x,\omega)\in L^2(\Omega)$ and $\psi\in H^s(\mathbb{R}^n)$. We say that $u_{\omega}\in H^s(\mathbb{R}^n)$ is a weak solution of \eqref{eq:fractional3} if and only if $u_{\omega}=u_{\omega}^{(0)}+u_{\omega}^{(\psi)}$,
	where $u_{\omega}^{(\psi)}\in H^s(\mathbb{R}^n)$ fulfils $u_{\omega}^{(\psi)}|_{\Omega_e}=\psi$ (especially when $\psi\in C_c^{\infty}(\Omega_e)$, we can choose $u_{\omega}^{(\psi)}:=\psi$), and $u_{\omega}^{(0)}\in H^s_0(\Omega)$ solves  
	\begin{equation}\label{weaksolution}
	\mathscr{B}_{\omega,q}(u_{\omega}^{(0)}, \phi) = -\mathscr{B}_{\omega,q}(u_{\omega}^{(\psi)}, \phi)+ \int_{\Omega}p(x,\omega)\phi(x)dx \ \mbox{ for any } \phi(x)\in C^{\infty}_c(\Omega),
	\end{equation}
	where $u_{\omega}$ denotes $u(x,\omega)$ with $\omega\in (0, \omega_0)$.
\end{defn}

\begin{rem}\label{closure}
	Since $\tilde{H}^s(\Omega):=$ closure of $C_c^{\infty}(\Omega)$ in $H^s(\mathbb{R}^n)$, by the standard density argument, all the test functions used in \eqref{weaksolution} can be replaced by the functions in $\tilde{H}^s(\Omega)$.
\end{rem}
The well-posedness of the nonlocal Dirichlet problem \eqref{eq:fractional3} can be expressed by the following lemma.
\begin{lem}\label{lem:lll1}
	Let $\Omega$ be a bounded Lipschitz domain. Suppose $q(x)\in L^{\infty}(\Omega)$, $p(x,\omega)\in L^2(\Omega)$ for any fixed $\omega\in(0, \omega_0)$. $u_{\omega}$ solves $L_{\sigma}^su_{\omega}+\omega^2qu_{\omega}=p(x,\omega)$ in $\Omega$ uniquely if and only if $u_{\omega}\in H^s(\mathbb{R}^n)$ satisfies
	\begin{equation}\label{well-posedness}
		\mathscr{B}_{\omega,q}(u_{\omega}, v) =\int_{\Omega}p(x,\omega)v(x)dx \quad\mbox{ for any } v(x)\in \tilde{H}^s(\Omega),
	\end{equation}
	and once $q(x)$ satisfies the eigenvalue condition \eqref{condition}, we can also obtain that $u_{\omega}$ has the following stability estimate
	\begin{equation}\label{estimate2}
	\lVert u_{\omega}\rVert_{H^s(\mathbb{R}^n)}\leq C(\lVert p(x,\omega)\rVert_{L^2(\Omega)}+\lVert \psi\rVert_{H^s(\mathbb{R}^n)}),
	\end{equation}
	for some positive constant C independent of $p$ and $\psi$ .
\end{lem}
The proof of this lemma follows from a similar argument to \cite[Lemma 3.1]{Cao17}. We also refer the readers \cite[Lemma 2.3]{Mikko16} as well as \cite[Proposition 3.3]{Ghosh17}. For self-containedness, we next sketch the proof of Lemma~\ref{lem:lll1}.
\begin{proof}[Proof of Lemma~\ref{lem:lll1}]
We have
	\begin{align*}
		&\int_{\Omega}(L^s_{\sigma}u_{\omega}+\omega^2q(x)u_{\omega}- p(x,\omega))v(x)dx\\&=\int_{\mathbb{R}^n\times\mathbb{R}^n}(u_{\omega}(x)-u_{\omega}(y))(v(x)-v(y))\mathscr{K}_s(x,y)dxdy\\&+\omega^2\int_{\Omega}q(x)u_{\omega}(x)v(x)dx-\int_{\Omega}p(x,\omega)v(x)dx
	\end{align*}
	for all $v(x)\in \tilde{H}^s(\Omega)$. By the definition of the bilinear form $\mathscr{B}_{\omega,q}$, one can directly verify \eqref{well-posedness} by straightforward calculations. As for the stability estimate result \eqref{estimate2}, since by \eqref{estimate} we know that $\mathscr{B}_{\omega,q}$ is bounded, coercive and continuous, thus we can use Lax-Milgram theorem (cf. \cite[Section 3]{Ghosh17}) to complete the proof.	
\end{proof}
\begin{rem}
	The solution $u_{\omega}\in H^s(\mathbb{R}^n)$ of \eqref{eq:fractional3} is independent of the value $\psi(x)\in H^s(\mathbb{R}^n)$ in $\Omega$, and it only depends on the value $\psi(x)\in H^s(\mathbb{R}^n)$ in $\Omega_e = \mathbb{R}^n\backslash\overline{\Omega}$. Thus in the subsequent study, we consider $\psi(x)\in \mathbb{T}$ with
	\begin{equation*}
	\mathbb{T}:=H^s(\mathbb{R}^n)/\tilde{H}^s(\Omega)
	\end{equation*}
	 the abstract quotient space.
	
	Since $\Omega\subset\mathbb{R}^n$ is a bounded Lipschitz domain, the quotient space $\mathbb{T}$ is isometric to $H^s(\Omega_e)$.
\end{rem}

\subsection{The Dirichlet to Neumann map}

The Dirichlet to Neumann map associated with the nonlocal problem \eqref{eq:fractional3} can also be defined via the bilinear form $\mathscr{B}_{\omega, q}$ as follows. 
\begin{prop}
	For $n\geq2, 0<s<1$, let $\Omega$ be a bounded Lipschitz domain. Suppose $q(x)\in L^{\infty}(\Omega)$ satisfies condition \eqref{condition} and $\omega\in (0,\omega_0)$. $\mathbb{T}:=H^s(\mathbb{R}^n)/\tilde{H}^s(\Omega)$ is the abstract quotient space introduced above. Define
	\begin{equation}
	(\Lambda_{q,p}^{\omega}\psi,h)_{\mathbb{T}^*\times\mathbb{T}}:=\mathscr{B}_{\omega,q}(u_{\omega,\psi},h) \quad\mbox{ for any } \psi, h\in \mathbb{T}
	\end{equation}
	where $u_{\omega,\psi}$ is the solution to the nonlocal Dirichlet problem \eqref{eq:fractional3} with the associated exterior Dirichlet value $\psi$. Then we define the map $\Lambda_{q,p}^{\omega}: \mathbb{T}\rightarrow\mathbb{T}^*$ as the exterior Dirichlet to Neumann (DtN) map. Moreover, $\Lambda_{q,p}^{\omega}$ has the following properties:
	 
	 (a). The DtN map $\Lambda_{q,p}^{\omega}$ is a bounded linear map.
	 
	 (b). $(\Lambda_{q,p}^{\omega}\psi,h)_{\mathbb{T}^*\times\mathbb{T}}=(\Lambda_{q,p}^{\omega}h,\psi)_{\mathbb{T}^*\times\mathbb{T}}$ for any $\psi,h\in \mathbb{T}$.
\end{prop}
\begin{proof}
	The proof of this proposition follows from a similar argument to \cite[Proposition 3.5]{Ghosh17}. 
\end{proof}

\begin{rem}
	Suppose $\tilde{h}$ is an arbitrary extension of $h\in\mathbb{T}$ in $\mathbb{R}^n$, $\omega\in (0,\omega_0)$. By the definition of the DtN map as well as the bilinear form we know that
\begin{align*}
(\Lambda_{q,p}^{\omega}\psi,h)_{\mathbb{T}^*\times\mathbb{T}}&=\mathscr{B}_{\omega,q}(u_{\omega,\psi},\tilde{h})\\
&=\int_{\mathbb{R}^n}(L^s_{\sigma}u_{\omega,\psi})\tilde{h}\, dx+\omega^2\int_{\Omega}qu_{\omega,\psi}\tilde{h}\, dx\\
&=\int_{\Omega_e}(L^s_{\sigma}u_{\omega,\psi})\tilde{h}\, dx\\
&=\int_{\Omega_e}(L^s_{\sigma}u_{\omega,\psi})h\,dx,
\end{align*}
thus we can conclude that
\begin{equation}\label{equi DtN}
\Lambda_{q,p}^{\omega}(\psi)=L^s_{\sigma}u_{\omega,\psi}|_{\Omega_e}=(-\nabla\cdot(\sigma\nabla))^su_{\omega, \psi}|_{\Omega_e}.
\end{equation}
\end{rem}

\section{Low-frequency asymptotics}
In order to study the inverse problem \eqref{eq:ip1}, and show the unique determination results for the corresponding source term $p(x,\omega)$ as well as the medium parameter $q(x)$, we derive in this section the crucial low-frequency asymptotics of the wave field. We shall need the following compact embedding theorem for the fractional Sobolev space. 

\begin{lem}[Compact embedding theorem]\label{embedding}
	For $n\geq2$, $0<s<1$, let $\Omega$ be a bounded Lipschitz domain and $H^s(\Omega)$ be a bounded subset of $L^2(\Omega)$. Suppose that
	\begin{equation*}
	\sup_{v\in H^s(\Omega)}\int_{\Omega\times\Omega}\frac{|v(x)-v(y)|^2}{|x-y|^{n+2s}}dxdy<+\infty.
	\end{equation*}
	Then $H^s(\Omega)$ is precompact in $L^2(\Omega)$.
	
\end{lem}
\begin{rem}
	This lemma is a direct consequence of the compact embedding theorem in \cite[Lemma 10]{Savin11}. The readers can also consult \cite[Theorem 7.1]{Nezza12} for the proof of the theorem for the more general fractional Sobolev space $W^{s,p}(\Omega)$, with $s\in(0,1)$ and $p\in[1,+\infty)$. 
\end{rem}

In the rest of the analysis, we assume that $p(x,\omega)$ satisfies the following regularity condition: $p(x,\omega)\in C^3[0,\omega_0]$ with respect to $\omega$ for any $x\in\Omega$. Hence, $p(x,\omega)$ has the following Taylor's series expansion in a sufficiently small neighbourhood of $\omega=0$,
\begin{equation}\label{analytic form}
p(x,\omega)=p(x,0)+\frac{\partial p(x,\omega)}{\partial\omega}\bigg|_{\omega=0}\omega+\frac{1}{2!}\frac{\partial^2p(x,\omega)}{\partial\omega^2}\bigg|_{\omega=0}\omega^2+\frac{1}{3!}\frac{\partial^3 p(x,\omega)}{\partial\omega^3}\bigg|_{\omega=\theta}\omega^3,
\end{equation}
where $\theta\in(0,\omega)$. Based on \eqref{analytic form}, we define
\begin{equation*}
 \tilde{p}(x,\omega):=\sum_{n=1}^{2}\frac{1}{n!}\frac{\partial^n p(x,\omega)}{\partial\omega^n}\bigg|_{\omega=0}\omega^n+\frac{1}{3!}\frac{\partial^3p(x,\omega)}{\partial\omega^3}\bigg|_{\omega=\theta}\omega^3, \quad \theta\in(0,\omega),
\end{equation*}
and then
\begin{equation}\label{higher order}
p(x,\omega)=p(x,0)+\tilde{p}(x,\omega).
\end{equation}

We are in a position to present the low-frequency asymptotic expansion of the wave field. 

\begin{thm}\label{asym-prop}
	For $n\geq2, 0<s<1$, let $\Omega\subset\mathbb{R}^n$ be a bounded Lipschitz domain. Suppose $q(x)\in L^{\infty}(\Omega)$ satisfies condition \eqref{condition} and $p(x,\omega)\in L^2(\Omega)$ for any fixed $\omega\in(0,\omega_0)$. Then for any exterior Dirichlet data $\psi\in \mathbb{T}$, the solution $u(x,\omega)$ of the Dirichlet problem \eqref{eq:fractional3} with respective to $\psi$ fulfils the following asymptotic estimate in $H^s(\Omega)$,
	\begin{equation}\notag
    \lVert u(x,\omega)-u(x,0)\rVert_{H^s(\Omega)}\leq\tilde{C}(\omega)\left(\frac{1}{\alpha_0}\lVert \tilde{p}(x,\omega)\rVert_{L^2(\Omega)}^2+\omega^2\lVert q\rVert_{L^{\infty}(\Omega)}\lVert u(x,0)\rVert_{L^2(\Omega)}^2\right)^{\frac{1}{2}},
	\end{equation} 
	where $u(x,0)$ solves
	\begin{equation}\label{eq:u_0}
	\begin{cases}
	(-\nabla\cdot(\sigma\nabla))^su(x,0)=p(x,0) \quad \mbox{ in }\ \ \Omega,\medskip\\
	u(x,0)=\psi(x) \quad \mbox{ in }\ \ \Omega_e=\mathbb{R}^n\backslash\overline{\Omega},
	\end{cases}
	\end{equation}
	and
	\begin{equation*}
	\tilde{C}(\omega)=\frac{1}{\sqrt{\frac{2c_0^2c_1}{(1+c_0)^2}-(\alpha_0+3\omega^2\lVert q\rVert_{L^{\infty}(\Omega)})}},
	\end{equation*}
	in which $c_0, c_1, \alpha_0$ are three positive constants with $0<\alpha_0<\frac{2c_0^2c_1}{(1+c_0)^2}$. 
\end{thm}

\begin{proof}
	Denote $u(x,\omega)$ as $u_{\omega}$ and $u(x,0)$ as $u_0$. For the problem \eqref{eq:fractional3}, by the definition of weak solution \eqref{weaksolution}, we know that for $p(x,\omega)\in L^2(\Omega)$ with a fixed $\omega\in(0,\omega_0)$, and $\psi\in\mathbb{T}$, we have 
	\begin{equation*}
	\mathscr{B}_{\omega,q}(u_{\omega}^{(0)}, v) = -\mathscr{B}_{\omega,q}(u_{\omega}^{(\psi)}, v)+ \int_{\Omega}p(x,\omega)v(x)\,dx \ \mbox{ for any } v(x)\in C^{\infty}_c(\Omega),
	\end{equation*}
	where $u_{\omega}^{(\psi)}\in H^s(\mathbb{R}^n)$ fulfils $u_{\omega}^{(\psi)}|_{\Omega_e}=\psi$, $u_{\omega}^{(0)}\in H^s_0(\Omega)$ and $u_{\omega}^{(0)}+u_{\omega}^{(\psi)}=u_{\omega}$. Equivalently, 
	if $u_{\omega}-\psi\in\tilde{H}^s(\Omega)$, we have
	\begin{equation}\label{eq:trans1}
	\mathscr{B}_{\omega,q}(u_{\omega},v)=\int_{\Omega}p(x,\omega)v(x)\, dx \ \mbox{ for any }v(x)\in C_c^{\infty}(\Omega).
	\end{equation}
	This means that for any $v(x)\in C_c^{\infty}(\Omega)$,  by \eqref{def:bilinear}
	\begin{align}\label{norm-asym1}
	&\int_{\mathbb{R}^n}(-\nabla\cdot(\sigma\nabla))^su_{\omega}(x)v(x)dx+\omega^2\int_{\Omega}q(x)u_{\omega}(x)v(x)dx\notag\\
	&=\int_{\mathbb{R}^n\times\mathbb{R}^n}(u_{\omega}(x)-u_{\omega}(y))(v(x)-v(y))\mathscr{K}_s(x,y)\,dxdy+\omega^2\int_{\Omega}q(x)u_{\omega}(x)v(x)\,dx\notag\\
    &=\int_{\Omega}p(x,\omega)v(x)\,dx.
	\end{align}
	Similarly, for \eqref{eq:u_0}, when $u_0(x)-\psi(x)\in \tilde{H}^s(\Omega)$, we have for any $v(x)\in C_c^{\infty}(\Omega)$,
	\begin{align}\label{norm-asym2}
	\mathscr{B}_{\omega,q}(u_0,v)&=
	\int_{\mathbb{R}^n}(-\nabla\cdot(\sigma\nabla))^su_0(x)v(x)\,dx\notag\\
	&=\int_{\mathbb{R}^n\times\mathbb{R}^n}(u_0(x)-u_0(y))(v(x)-v(y))\mathscr{K}_s(x,y)\,dxdy\\
	&=\int_{\Omega}p(x,0)v(x)\,dx. \notag
	\end{align}
	Substracting \eqref{norm-asym2} from \eqref{norm-asym1}, we can obtain the following equation
	\begin{align*}
	&\int_{\mathbb{R}^n\times\mathbb{R}^n}[(u_{\omega}(x)-u_0(x))-(u_{\omega}(y)-u_0(y))](v(x)-v(y))\mathscr{K}_s(x,y)\, dxdy\\
	&+\omega^2\int_{\Omega}q(x)u_{\omega}(x)v(x)\, dx\\
	&=\int_{\Omega}(p(x,\omega)-p(x,0))v(x)\, dx=\int_{\Omega}\tilde{p}(x,\omega)v(x)\, dx. 
	\end{align*}
	Define $h_{\omega}(x):=u_{\omega}(x)-{u}_0(x)$, then by direct computations we have
	\begin{align*}
	&\int_{\mathbb{R}^n\times\mathbb{R}^n}(h_{\omega}(x)-h_{\omega}(y))(v(x)-v(y))\mathscr{K}_s(x,y)\,dxdy+\omega^2\int_{\Omega}q(x)h_{\omega}(x)v(x)\,dx\\
	&=\int_{\Omega}\tilde{p}(x,\omega)v(x)dx-\omega^2\int_{\Omega}q(x)u_0(x)v(x)\,dx  \quad\mbox{ for any }v(x)\in C_c^{\infty}(\Omega),
	\end{align*}
	and thus
	\begin{align}\label{norm-asym3}
	&\int_{\mathbb{R}^n\times\mathbb{R}^n}(h_{\omega}(x)-{h}_{\omega}(y))(v(x)-v(y))\mathscr{K}_s(x,y)\,dxdy\notag\\
	&=\int_{\Omega}\tilde{p}(x,\omega)v(x)dx-\omega^2\int_{\Omega}q(x){u}_0(x)v(x)\,dx-\omega^2\int_{\Omega}q(x)h_{\omega}(x)v(x)\,dx   
	\end{align}
	for any $v(x)\in C_c^{\infty}(\Omega)$.
	
	Since $u_{\omega}, u_0\in H^s(\mathbb{R}^n)$, we see that $h_{\omega}\in H^s(\mathbb{R}^n)$. Furthermore, we know that
	\begin{equation*}
	h_{\omega}(x)|_{\Omega_e}=u_{\omega}(x)|_{\Omega_e}-u_0(x)|_{\Omega_e}=\psi(x)-\psi(x)=0,
	\end{equation*} 
	and hence $h_{\omega}\in \tilde{H}^s(\Omega)$. Using Remark \ref{closure}, we can assume that $v(x)={h}_{\omega}(x)$ and substitute this into \eqref{norm-asym3} we can obtain
	\begin{align}\label{norm-asym4}
	&\int_{\mathbb{R}^n\times\mathbb{R}^n}({h}_{\omega}(x)-h_{\omega}(y))^2\mathscr{K}_s(x,y)\, dxdy\notag\\
	&=\int_{\Omega}\tilde{p}(x,\omega){h}_{\omega}(x)dx-\omega^2\int_{\Omega}q(x){u}_0(x){h}_{\omega}(x)\,dx-\omega^2\int_{\Omega}q(x){h}^2_{\omega}(x)\,dx  .
	\end{align}
	Since the operator kernel has the pointwise estimate \eqref{estimate}, the left hand side of \eqref{norm-asym4} satisfies
	\begin{align}\label{lhs-estimate}
	c_1\int_{\mathbb{R}^n\times\mathbb{R}^n}\frac{|{h}_{\omega}(x)-{h}_{\omega}(y)|^2}{|x-y|^{n+2s}}dxdy&
	\leq\int_{\mathbb{R}^n\times\mathbb{R}^n}({h}_{\omega}(x)-{h}_{\omega}(y))^2\mathscr{K}_s(x,y)dxdy\notag\\ 
	&\leq c_2\int_{\mathbb{R}^n\times\mathbb{R}^n}\frac{|{h}_{\omega}(x)-{h}_{\omega}(y)|^2}{|x-y|^{n+2s}}dxdy
	\end{align}
	for some constants $c_1, c_2>0$.
	
	It is known that for $s\in(0,1)$, $\lVert\cdot\rVert_{H^s(\mathbb{R}^n)}$ has the form \eqref{normdef} with the semi-norm defined as \eqref{seminorm} in any open set $\Omega\subset\mathbb{R}^n$, so we have
	\begin{align}\label{lhs-estimate2}
	c_1|{h}_{\omega}|^2_{H^s(\Omega)}&=c_1\int_{\Omega\times\Omega}\frac{|{h}_{\omega}(x)-{h}_{\omega}(y)|^2}{|x-y|^{n+2s}}dxdy
	\leq c_1\int_{\mathbb{R}^n\times\mathbb{R}^n}\frac{|{h}_{\omega}(x)-{h}_{\omega}(y)|^2}{|x-y|^{n+2s}}dxdy\notag\\
	&\leq\int_{\mathbb{R}^n\times\mathbb{R}^n}({h}_{\omega}(x)-{h}_{\omega}(y))^2\mathscr{K}_s(x,y)dxdy,
	\end{align}	
	where the second inequality is obtained by \eqref{lhs-estimate}. Moreover, since $\Omega$ is a bounded Lipschitz domain, utilizing the compact embedding theorem (Lemma \ref{embedding}) in the fractional Sobolev space $H^s(\Omega)$, we can have that there exists a positive constant $c_0$ such that
	\begin{equation*}
	c_0\lVert h_{\omega}\rVert_{L^2(\Omega)}\leq|h_{\omega}|_{H^s(\Omega)},
	\end{equation*}
	and thus
	\begin{equation}\label{embedding2}
	c_0\lVert h_{\omega}\rVert_{H^s(\Omega)}\leq(1+c_0)|h_{\omega}|_{H^s(\Omega)}
	\end{equation}
	
	Next, by combining \eqref{norm-asym4}, \eqref{lhs-estimate2} and \eqref{embedding2}, we have the following estimate
	\begin{align}\label{inequality1}
	& \frac{c_0^2c_1}{(1+c_0)^2}\lVert h_{\omega}\rVert_{H^s(\Omega)}^2\leq c_1\lvert h_{\omega}\rvert_{H^s(\Omega)}^2\notag\\
	&\leq |\int_{\mathbb{R}^n\times\mathbb{R}^n}({h}_{\omega}(x)-{h}_{\omega}(y))^2\mathscr{K}_s(x,y)dxdy|\notag\\
	&=|\int_{\Omega}\tilde{p}(x,\omega){h}_{\omega}(x)dx-\omega^2\int_{\Omega}q(x){u}_0(x){h}_{\omega}(x)dx-\omega^2\int_{\Omega}q(x){h}^2_{\omega}(x)dx|\notag\\
	&\leq
    \lvert\int_{\Omega}\tilde{p}(x,\omega){h}_{\omega}(x)dx\rvert+\omega^2\lVert q\rVert_{L^{\infty}(\Omega)}\lvert\int_{\Omega}u_0(x){h}_{\omega}(x)dx\rvert\notag\\
	&+\omega^2\lVert q\rVert_{L^{\infty}(\Omega)}\int_{\Omega}h^2_{\omega}(x)dx
	\end{align}
	Choosing a constant $\alpha_0$ in the interval $\bigg(0,\frac{2c_0^2c_1}{(1+c_0)^2}\bigg)$ and by utilizing Young's inequality with $\alpha_0$ we have
	\begin{equation}\label{ineq:p}
	\begin{split}
	&\left\lvert\int_{\Omega}\tilde{p}(x,\omega){h}_{\omega}(x)dx\right\rvert\\
	\leq & \frac{1}{2}\int_{\Omega}\left(\frac{\tilde{p}(x,\omega)^2}{\alpha_0}+\alpha_0h^2_{\omega}(x)\right)dx\\
	=&\frac{1}{2\alpha_0}\lVert \tilde{p}(x,\omega)\rVert_{L^2(\Omega)}^2+\frac{\alpha_0}{2}\lVert h_{\omega}\rVert_{L^2(\Omega)}^2.
	\end{split}
	\end{equation}
	Similarly,
	\begin{equation}\label{ineq:u_0}
	\left\lvert \int_{\Omega}u_0(x)h_{\omega}(x)dx\right\rvert\leq\frac{1}{2}\left(\lVert u_0\rVert_{L^2(\Omega)}^2+\lVert h_{\omega}\rVert_{L^2(\Omega)}^2\right).
	\end{equation}
	
	Plugging \eqref{ineq:p} and \eqref{ineq:u_0} into \eqref{inequality1}, and using the fact that $\lVert h_{\omega}\rVert_{L^2(\Omega)}\leq\lVert h_{\omega}\rVert_{H^s(\Omega)}$, it is clear that
	\begin{align*}
	\frac{c_0^2c_1}{(1+c_0)^2}\lVert h_{\omega}\rVert_{H^s(\Omega)}^2
	&\leq \frac{1}{2\alpha_0}\lVert \tilde{p}(x,\omega)\rVert_{L^2(\Omega)}^2+\frac{\omega^2}{2}\lVert q\rVert_{L^{\infty}(\Omega)}\lVert u_0\rVert_{L^2(\Omega)}^2\\\notag
	&+(\frac{\alpha_0}{2}
	+\frac{3\omega^2}{2}\lVert q\rVert_{L^{\infty}(\Omega)})\lVert h_{\omega}\rVert_{H^s(\Omega)}^2. 
	\end{align*}
	Thus we have proved that
	\begin{equation}\label{asymp estimate}
	\begin{split}
	&\left( \frac{c_0^2c_1}{(1+c_0)^2}-(\frac{\alpha_0}{2}
	+\frac{3\omega^2}{2}\lVert q\rVert_{L^{\infty}(\Omega)})\right)\lVert h_{\omega}\rVert_{H^s(\Omega)}^2\\
	\leq & \frac{1}{2\alpha_0}\lVert \tilde{p}(x,\omega)\rVert_{L^2(\Omega)}^2+\frac{\omega^2}{2}\lVert q\rVert_{L^{\infty}(\Omega)}\lVert u_0\rVert_{L^2(\Omega)}^2
	\end{split}
	\end{equation}
	Define $\tilde{C}(\omega)=\frac{1}{\sqrt{\frac{2c_0^2c_1}{(1+c_0)^2}-(\alpha_0+3\omega^2\lVert q\rVert_{L^{\infty}})}}$. Then by straightforward computations, we have
	\begin{equation*}
	\begin{split}
	& \lVert u(x,\omega)-u(x,0)\rVert_{H^s(\Omega)}=\lVert h_{\omega}(x)\rVert_{H^s(\Omega)}\\
	\leq& \tilde{C}(\omega)\left(\frac{1}{\alpha_0}\lVert \tilde{p}(x,\omega)\rVert_{L^2(\Omega)}^2+\omega^2\lVert q\rVert_{L^{\infty}(\Omega)}\lVert u(x,0)\rVert_{L^2(\Omega)}^2\right)^{\frac{1}{2}},
	\end{split}
	\end{equation*}
	which completes the proof.
\end{proof}

\begin{cor}\label{asymp-limit}
	Consider the estimate \eqref{asymp estimate} in the previous proof, since $c_0, c_1>0$, $q\in L^{\infty}(\Omega)$, $\tilde{p}(x,\omega)\in L^2(\Omega)$ for any $w\in (0,\omega_0)$ and $u_0(x)\in H^s(\Omega)$ which clearly indicates that $u_0(x)\in L^2(\Omega)$, we can take the limit $\omega\rightarrow0$ and hence
	\begin{equation}\label{eq:limit1}
	0\leq \left( \frac{c_0^2c_1}{(1+c_0)^2}-\frac{\alpha_0}{2} \right)\lVert h_{\omega}\rVert_{H^s(\Omega)}^2\xrightarrow{\omega\rightarrow0}0 ,
	\end{equation}
	where the left inequality holds due to $0<\alpha_0<\frac{2c_0^2c_1}{(1+c_0)^2}$, and the limit $0$ exists because of the property that 
	\begin{equation*}
	\lVert \tilde{p}(x,\omega)\rVert_{L^2(\Omega)}=\left\lVert\sum_{n=1}^{2}\frac{1}{n!}\frac{\partial^n p(x,\omega)}{\partial\omega^n}\bigg|_{\omega=0}\omega^n+\frac{1}{3!}\frac{\partial^3p(x,\omega)}{\partial\omega^3}\bigg|_{\omega=\theta}\omega^3\right\rVert_{L^2(\Omega)}\xrightarrow{\omega\rightarrow0}0, 
	\end{equation*}
	where $\theta\in(0,\omega)$. Therefore we can deduce that $\lVert u(x,\omega)-u(x,0)\rVert_{H^s(\Omega)}\rightarrow 0$ as $\omega\rightarrow 0$.
\end{cor}

\begin{rem}
We believe that the mathematical strategy developed in proving Theorem~\ref{asym-prop} by combining the variational argument and the compact embedding theorem can be used to deal with other nonlocal problems in different contexts. 
\end{rem}

\section{Unique determination results}\label{sect:5}

With the earlier preparations, we are ready to show the uniqueness in determining the internal source $p$, namely Theorem~\ref{thm:1}, and the medium parameter $q$, namely Theorem~\ref{thm:2}, by the corresponding exterior measurements. The following strong uniqueness principle as well as the Runge approximation property shall be needed in the discussion.

\begin{prop}[Strong uniqueness principle]\label{strong unique}
	For $n\geq2 $ and $0<s<1$, if $u\in H^s(\mathbb{R}^n)$ satisfies $u=L^s_{\sigma}u=(-\nabla\cdot(\sigma\nabla))^su=0$ in any nonempty open subset of $\mathbb{R}^n$, then $u\equiv0$ in $\mathbb{R}^n$. 
\end{prop}
 See \cite[Theorem 1.2]{Ghosh17} for the analysis of this proposition.
 
 \begin{prop}[Runge approximation property]\cite[Theorem 1.3]{Ghosh17} \label{Runge approxiamtion}
 	For $n\geq2$ and $0<s<1$, let $\Omega\subset\mathbb{R}^n$ be a bounded open set and $D\subseteq\mathbb{R}^n$ be an arbitrary open set containing $\Omega$ such that $int(D\backslash\overline{\Omega})\neq\emptyset$. Suppose the condition \eqref{condition} is fulfilled. Then for any $f \in L^2(\Omega)$ and $\varepsilon>0$, we can find a function $u_{\varepsilon}$ which solves 
 	\begin{equation*}
 	\begin{cases}
 		(L^s_{\sigma}+q)u_{\varepsilon}=0 \mbox{ in }\Omega \mbox{ and } supp(u_{\varepsilon})\subseteq\overline{D} ,	\\
 		u_\varepsilon=\psi \mbox{ in } \Omega_e,
 	\end{cases}
 	\end{equation*}
 	with $u_\varepsilon-\psi\in \tilde{H}^s(\Omega)$, where $L^s_{\sigma}=(-\nabla\cdot(\sigma\nabla))^s$ is the fractional differential Laplacian operator defined earlier, such that
 	\begin{equation*}
 	\lVert u_{\varepsilon}-f\rVert_{L^2(\Omega)}<\varepsilon.
 	\end{equation*} 
 \end{prop}


\begin{proof}[Proof of Theorem \eqref{thm:1}]
	Consider the nonlocal Dirichlet problems for the fractional Helmholtz system with respect to $q_j$, $p_j$, $j=1,2$ as
	\begin{equation*}
	\begin{cases}
	(-\nabla\cdot(\sigma\nabla))^su_j(x,\omega)+\omega^2q_j(x)u_j(x,\omega)=p_j(x,\omega) \quad\mbox{ in }\ \ \Omega,\medskip\\
	u_j(x,\omega)=\psi(x)\quad \mbox{ in }\ \ \Omega_e. 
	\end{cases}
	\end{equation*}
	Define $W(x,\omega):=u_1(x,\omega)-u_2(x,\omega)$. Then $W(x,\omega)\in H^s(\mathbb{R}^n)$ with $\omega\in (0,\omega_0)$ and it satisfies
	\begin{equation}\label{eq:W}
	\begin{cases}
	(-\nabla\cdot(\sigma\nabla))^sW(x,\omega)+\omega^2q_1W(x,\omega)\medskip\\
	=(p_1(x,\omega)-p_2(x,\omega))+\omega^2(q_2-q_1)u_2(x,\omega)\ \mbox{ in }\ \Omega,\medskip\\
	W(x,\omega)=0 \ \mbox{ in }\ \Omega_e.
	\end{cases}
	\end{equation}
	If for any fixed $\psi\in C_c^{\infty}(\mathcal{O}_1)$, there holds
	\begin{equation*}
	\Lambda_{q_1,p_1}^{\omega}(\psi)|_{\mathcal{O}_2}=	\Lambda_{q_2,p_2}^{\omega}(\psi)|_{\mathcal{O}_2},
	\end{equation*}
	by \eqref{equi DtN} which means
	\begin{equation*}
	(-\nabla\cdot(\sigma\nabla))^su_1(x,\omega)|_{\mathcal{O}_2}=(-\nabla\cdot(\sigma\nabla))^su_2(x,\omega)|_{\mathcal{O}_2} \ \mbox{ for any fixed }\psi\in C_c^{\infty}(\mathcal{O}_1),
	\end{equation*}
	where $\mathcal{O}_1, \mathcal{O}_2\subset\Omega_e$ are two arbitrary nonempty open subsets, then
	\begin{equation*}
	(-\nabla\cdot(\sigma\nabla))^sW(x,\omega)|_{\mathcal{O}_2}=(-\nabla\cdot(\sigma\nabla))^su_1(x,\omega)|_{\mathcal{O}_2}-	(-\nabla\cdot(\sigma\nabla))^su_2(x,\omega)|_{\mathcal{O}_2}=0.
	\end{equation*}
	So we have
	\begin{equation*}
	W(x,\omega)=(-\nabla\cdot(\sigma\nabla))^sW(x,\omega)=0 \quad \mbox{ in }\ \ \mathcal{O}_2\subset\Omega_e. 
	\end{equation*}
	By applying the strong uniqueness principle (Proposition \ref{strong unique}), we obtain that $W(x,\omega)\equiv0$ in $\mathbb{R}^n$. Going back to the first equation in \eqref{eq:W}, we have for $\omega\in(0,\omega_0)$
	\begin{equation}\label{eq:equiv}
	p_2(x,\omega)-p_1(x,\omega)=\omega^2(q_2(x)-q_1(x))u_2(x,\omega),
	\end{equation}
	with 
	\begin{align*}
	p_j(x,\omega)&=p_j(x,0)+\tilde{p}_j(x,\omega)\\
	&=p_j(x,0)+\sum_{n=1}^{2}\frac{1}{n!}\frac{\partial^np_j(x,\omega)}{\partial \omega^n}\bigg|_{\omega=0}\omega^n+\frac{1}{3!}\frac{\partial^3p_j(x,\omega)}{\partial\omega^3}\bigg|_{\omega=\theta_j}\omega^3,
	\end{align*}
	where $\theta_j\in(0,\omega)$, for $j=1,2$. Utilizing the low-frequency asymptotic property (Corollary \ref{asymp-limit}), we know that $u_2(x,\omega)\xrightarrow{\omega\rightarrow0}u_2(x,0)$ in $H^s(\Omega)$ where $u_2(x,0)$ solves
	\begin{equation*}
	\begin{cases}
	(-\nabla\cdot(\sigma\nabla))^su_2(x,0)=p_2(x,0)\quad\mbox{ in }\ \ \Omega,\medskip\\
	u_2(x,0)=\psi\quad\mbox{ in }\ \ \Omega_e.
	\end{cases}
	\end{equation*}
	Taking $\omega\rightarrow0$ in equation \eqref{eq:equiv}, we thus have
	\begin{align*}
	p_1(x,0)&=p_2(x,0),\\
	\frac{\partial p_1(x,\omega)}{\partial \omega}\bigg|_{\omega=0}&=\frac{\partial p_2(x,\omega)}{\partial \omega}\bigg|_{\omega=0},
	\end{align*}
	which complete the proof.
\end{proof}

Now we prove Theorem \ref{thm:2} on the basis of Theorem \ref{thm:1}.

\begin{proof}[Proof of Theorem \eqref{thm:2}]
	Since
	\begin{equation*}
	p_j(x,\omega)=\sum_{n=0}^{2}\frac{1}{n!}\frac{\partial^np_j(x,\omega)}{\partial \omega^n}\bigg|_{\omega=0}\omega^n+\frac{1}{3!}\frac{\partial^3p_j(x,\omega)}{\partial\omega^3}\bigg|_{\omega=\theta_j}\omega^3
	\end{equation*} 
	where $\theta_j\in(0,\omega)$ for $j=1,2$, based on the result of Theorem \ref{thm:1} that
	\begin{align*}
	p_1(x,0)&=p_2(x,0),\\
	\frac{\partial p_1(x,\omega)}{\partial \omega}\bigg|_{\omega=0}&=\frac{\partial p_2(x,\omega)}{\partial \omega}\bigg|_{\omega=0},
	\end{align*}
	and the assumption
	\begin{equation*}
	\frac{\partial^2(p_1-p_2)(x,\omega)}{\partial \omega^2}\bigg|_{\omega=0}=0\quad\mbox{ for }\ x\in\Omega,
	\end{equation*}
	we have by letting $\omega\rightarrow 0$ in the equation \eqref{eq:equiv} that
	\begin{equation}\label{eq:equiv2}
	(q_2(x)-q_1(x))u_2(x,0)=0.
	\end{equation}
	
	Suppose $\tilde{u}\in H^s(\mathbb{R}^n)$ solves the equation $(-\nabla\cdot(\sigma\nabla))^s\tilde{u}(x)=p_2(x,0)$ in $\Omega$ and $\tilde{u}(x)=0$ in $\Omega_e$, then $\tilde{u}\in\tilde{H}^s(\Omega)$ and we know that $\tilde{u}$ can actually be determined uniquely by the well-posedness discussion in Section 3.  Then we define $\tilde{v}(x):=u_2(x,0)-\tilde{u}(x)$, and it clearly satisfies 
	\begin{equation}\label{eq:L11}
	\begin{cases}
	(-\nabla\cdot(\sigma\nabla))^s\tilde{v}(x)=0 \quad\mbox{ in }\ \ \Omega,\medskip\\
	\tilde{v}(x)=\psi(x) \quad \mbox{ in }\ \ \Omega_e. 
	\end{cases}
	\end{equation}
	Then \eqref{eq:equiv2} can be reformulated by $\tilde{v}(x)$ and $\tilde{u}(x)$ as
	\begin{equation}\label{eq:equiv3}
	(q_2(x)-q_1(x))\tilde{v}(x)=-(q_2(x)-q_1(x))\tilde{u}(x),
	\end{equation}
	which holds for any $\tilde v$ satisfying \eqref{eq:L11}. 
	Since $\tilde{u}(x)\in \tilde{H}^s(\Omega)\subset L^2(\Omega)$, by the Runge approximation property in Proposition~\ref{Runge approxiamtion}, we know that there exists a sequence $(\tilde{v}^{(j)}(x))_{j=1}^{\infty}\in H^s(\Omega)$ which solves $$(-\nabla\cdot(\sigma\nabla))^s\tilde{v}^{(j)}(x)=0$$ in $\Omega$ with exterior values in $\Omega_e$ such that $\tilde{v}^{(j)}(x)\xrightarrow{j\rightarrow\infty}\tilde{u}(x)$ in $L^2(\Omega)$. By using the aforesaid sequence in the equation \eqref{eq:equiv3}, one has
	\begin{equation*}
	(q_2(x)-q_1(x))\tilde{v}^{(j)}(x)=-(q_2(x)-q_1(x))\tilde{u}(x),
	\end{equation*}
	and then by taking the limit  $j\rightarrow\infty$, one further has
	\begin{equation*}
	(q_2(x)-q_1(x))\tilde{u}(x)=0,
	\end{equation*}
	which by using \eqref{eq:equiv3} again readily implies that
	\begin{equation*}
	(q_2(x)-q_1(x))\tilde{v}(x)=0
	\end{equation*}
	for any $\tilde{v}$ satisfies \eqref{eq:L11}.
	Taking the integration of this equation in $\Omega$, we have
	\begin{equation}\label{eq:int}
	\int_{\Omega}(q_2(x)-q_1(x))\tilde{v}(x)dx=0
	\end{equation}
	Using the Runge approximation property again, we know that there exists a sequence $(\tilde{v}^{(l)}(x))_{l=1}^{\infty}\in H^s(\Omega)$ satisfying $(-\nabla\cdot(\sigma\nabla))^s\tilde{v}^{(l)}(x)=0$ in $\Omega$ with exterior values in $\Omega_e$ such that for any $f\in L^2(\Omega)$
	\begin{equation*}
	\tilde{v}^{(l)}(x)=f(x)+r^{(l)}(x) \quad\mbox{ with } \quad r^{(l)}(x)\xrightarrow{l\rightarrow\infty}0.
	\end{equation*}
	Substituting this formulation into the integral identity \eqref{eq:int} and letting $l\rightarrow\infty$ we can deduce that
	\begin{equation*}
	\int_{\Omega}(q_2(x)-q_1(x))f(x)dx=0
	\end{equation*}
	Since $f(x)\in L^2(\Omega)$ is arbitrary, we can infer that $q_1(x)=q_2(x)$.
	
	The proof is complete. 
\end{proof}

Finally, we consider the application of the unique recovery results in Theorems~\ref{thm:1} and \ref{thm:2} for the general fractional Helmholtz system~\eqref{eq:fractional3} to the specific one \eqref{eq:ff22} in simultaneously recovering the sound speed $c$ and the source terms $f, g$ and $\hat h$ by the exterior multiple-frequency measurements. We only point out two results of practical interest.  
First, we consider the case that $\hat h\equiv 0$. In such a case, $p(x,\omega)=\frac{\mathrm{i}\omega}{c^2} f(x)+\frac{1}{c^2}g(x)$ and $q(x)=-\frac{1}{c^2(x)}$. Hence, by Theorems~\ref{thm:1} and \ref{thm:2}, it can be easily shown that all of $f, g$ and $c$ can be uniquely recovered. For the second case, we consider that $f\equiv g\equiv 0$
whereas $\hat{h}(x,\omega)=\rho(x)\kappa(\omega)$. It is assumed that $\kappa(\omega)$ is continuous at $\omega=0$ and $\kappa(0)$ is nonzero and known a priori. Then by Theorem~\ref{thm:1} and also by checking its proof, one can prove that $\rho(x)$ can be uniquely recovered. Next, by Theorem 2, one can further prove that $c$ can be uniquely recovered as well. In the latter example $\hat{h}(x,\omega)$ is not necessary to fulfil the regularity assumption that it is $C^3$ continuous with respect to $\omega$ for any $x\in\Omega$. This example indicates that the mathematical techniques developed in this article can be used to establish more general simultaneous recovery results, and in particular, if there is some a priori knowledge available for the inverse problem.

\section*{Acknowledgement}

The authors would like to thank Dr. Yi-Hsuan Lin for helpful discussion. H. Liu was supported by the FRG and startup grants from Hong Kong Baptist University, and Hong Kong RGC General Research Funds, 12302415 and 12302017.

\end{document}